\documentclass{amsart}
\usepackage{amssymb, amsmath, dsfont}
\usepackage{amsaddr}
\usepackage{graphicx}
\usepackage{url} 
\usepackage{comment}
\usepackage{color}

\newcommand{\transp}[1]{#1^T}

\usepackage{comment}
\usepackage{algorithmic}
\usepackage{algorithm}
\usepackage{xspace}
\usepackage{bbold}
\usepackage{amsthm}
\usepackage{verbatim}
\usepackage{tikz}
\usetikzlibrary{arrows,shapes}
\newcommand{\diag}[1]{\operatorname{diag}(#1)}
\newcommand{\NN}{\mathds{N}}
\newcommand{\RR}{\mathds{R}}

\newcommand{\Id}{I}
\newcommand{\vect}[1]{\vec{#1}}
\newcommand{\gvect}[1]{\overset{\leftarrow}{#1}}
\newcommand{\spectre}{\operatorname{Sp}}

\newcommand{\ham}{\operatorname{Ham}}
\newcommand{\row}{\operatorname{row}}
\newcommand{\cm}{R}
\newcommand{\kernels}{\mathbb{K}}
\newcommand{\cubics}{\kernels^{(3)}}
\newcommand{\PATH}{P} 

\newcommand{\cms}{(\cm, \sigma)}

\newcommand{\ccm}{S}
\newcommand{\cgcm}{c}
\newcommand{\cc}{\operatorname{cc}}

\newcommand{\bin}[1]{[#1]_2}
\newcommand{\psat}{\mathds{P}_{\operatorname{SAT}}}
\newcommand{\probabip}{\mathds{P}_{\operatorname{Bip}}}
\newcommand{\smat}[1]{\left( \begin{smallmatrix} #1 \end{smallmatrix} \right)}

\newtheorem{theorem}{Theorem}
\newtheorem{lemma}[theorem]{Lemma}

\begin{document}

\title[Analytic Combinatorics of Inhomogeneous Multigraphs]{Analytic Description of the Phase Transition of Inhomogeneous Multigraphs}

\author{\'Elie de Panafieu}
\address{Research Institute for Symbolic Computation (RISC) \\
Johannes Kepler Universit\"at \\
Altenbergerstra\ss{}e 69 \\
A-4040 Linz, Austria}
\thanks{This work was partially founded by the ANR Boole, the ANR Magnum
and the Austrian Science Fund (FWF) grant F5004.\\ 
This is the long version of the homonym paper accepted in the proceedings of Eurocomb 2013.}

\author{Vlady Ravelomanana}
\address{Univ Paris Diderot, Sorbonne Paris Cit\'e,\\
  LIAFA, UMR 7089,\\
   75013, Paris, France}

\begin{abstract}

We introduce a new model of random multigraphs 
with colored vertices and weighted edges.
It is
similar to the \emph{inhomogeneous random graph}
model of S\"oderberg~\cite{S02}, extended 
by Bollob\'as, Janson and Riordan~\cite{BJR07}.
By means of analytic combinatorics,
we then analyze the birth of \emph{complex components},
which are components with at least two cycles.

We apply those results
to give a complete picture of the finite size scaling 
and the critical exponents associated 
to a rather broad family of decision problems.
As applications, we derive new proofs
of known results on the $2$-colorability problem~\cite{PY10}
and on the enumeration of properly $q$-colored multigraphs~\cite{W72}.
We also obtain new results on the phase transition of the satisfiability 
of quantified~$2$-Xor-formulas~\cite{DR11,CDE07}.

\end{abstract}

\maketitle

    \section{Introduction}

Phase transitions for Boolean Satisfiability (SAT) 
and for Constraint Satisfaction Problems (CSP)
are fundamental problems arising 
in different communities including 
computer science, mathematics and physics.
For any $k\geq 2$, the random version 
of the well-known $k$-SAT problem is known to exhibit 
a sharp phase transition~\cite{F99}: 
as the density $c$ of clauses
(where the number of clauses is $c$ times the number of variables) 
increases, the formula abruptly changes from being satisfiable to being unsatisfiable at a critical threshold point. 
For general CSP, the last decade has seen 
a growth of interest in determining the nature of the SAT/UNSAT phase transition (sharp or coarse), locating it, 
determining a precise scaling window 
and better understanding the structure of the space of solutions.
These turn out to be very challenging tasks 
(see e.g. \cite{DMSZ}, \cite{OM06}).

    \subsection{Related Works}

Recently, different authors gave precise descriptions of the phase transitions 
associated to several tractable decision problems.

The $2$-colorability problem consists in determining
if the vertices of a given graph can be colored with two colors 
in such a way that the vertices of each edge 
of the graphs have distinct colors.
Pittel and Yeum~\cite{PY10} derived the limit probability
for a random graph $G(n,m)$ with $n$ vertices and $m$ edges
to be $2$-colorable, when the parameter $m/n$ is smaller than 
or in the vicinity of~$1/2$.

Almost at the same time, 
Daud\'e and Ravelomanana~\cite{DR11} 
considered the $2$-Xor satisfiability problem
in which each instance is a formula 
that is a conjunction of Boolean equations of the 
form~$x \oplus y=0$ or~$x \oplus y= 1$.  
They showed that the probability that a random $2$-Xor formula 
is satisfiable tends to a similar threshold function.

In~\cite{CDE07}, Creignou~\textit{et al}
studied a quantified version of $2$-Xor-SAT.
They introduced $(a,2)$-QXor formula, which are
formula of the form~$\forall X \exists Y \phi(X, Y)$
where~$X$ and~$Y$ denote distinct set of variables 
and~$\phi$ is a Xor-formula with clauses containing 
exactly~$a$ variables from~$X$ and~$2$ variables from $Y$.
The problem consists in determining
if for \textit{every} assignment of the variables $X$, 
there \textit{exists} an assignment of the variables $Y$
such that $\phi(X,Y)$ is true. 
For any positive integers $a$, 
the authors showed that the phase transitions 
of~$(a, 2)$-QXor-SAT are coarse 
and that the probability of satisfiability is almost~$0$
when the number of variables is
around~$2$ times the number of clauses.

A random graph from the~$G(n, p)$ model
has~$n$ vertices and each pair of vertices
is linked with probability~$p$.
In~\cite{ER60}, Erd\H{o}s and R\'enyi
located the density of edges at which the first
connected component with more than one cycle
- called a \emph{complex} component - appears.
Using analytic tools, Janson, Knuth, \L{}uczak and Pittel
derived in~\cite{JKLP93} more information
on the structure of a random graph or multigraph near
the birth of complex components.
S\"oderberg introduced in~\cite{S02} a model
of inhomogeneous random graphs\footnote{We thank Annika Heckel for bringing to our notice the existence of this model.},
extended by Bollob\'as, Janson and Riordan~\cite{BJR07}.
This model generalizes~$G(n,p)$ in the following way:
each vertex receives a type among a set of~$q$ types,
and the probability that a vertex
of type~$i$ and one of type~$j$
are linked is the coefficient~$(i,j)$
of a symmetric matrix~$\cm$ of dimension~$q \times q$.
Among other results, they located the birth
of the complex component.
We combine here the accuracy of the approach of~\cite{JKLP93}
with the generality of the inhomogeneous random graph model.
We also enrich the model, adding a weight~$\sigma$ 
for each connected component.

    \subsection{Our Work}

Random Boolean formulas with two variables per clause 
can be modeled by random multigraphs.
Observe that the critical density~$1/2$,
common in~\cite{PY10,DR11} and~\cite{CDE07},
corresponds to an important change 
in the structure of the underlying graphs:
as the number of edges reaches half the number of vertices, 
components more complex than trees or unicyclics
start to appear in random graphs
(see for instance \cite{Bollobas,JKLP93}).
Our main goal is to establish 
a general framework that allows precise descriptions
of some of the phase transitions 
of random formulas with~$2$ variables per clause.
Namely our results apply to (and generalize) 
those in~\cite{PY10,DR11}. 
In particular, they give a more detailed picture 
of the transitions introduced in~\cite{CDE07} 
for quantified formulas.
To do so, we study a new model of colored and weighted random multigraphs,
similar to the inhomogeneous random graph model~\cite{S02} 
and to the multigraph process~\cite{JKLP93}.
We then propose a detailed analysis (thought still very general) 
with the purpose of quantifying 
the probability of satisfiability of formulas 
before and inside the critical window of their phase transitions.
Our work is based on generating functions and analytic methods. 

The rest of the paper is organized as follows: 
in Section~\ref{sec:model_theorem}, we first present our model
and derive the main theorem on the asymptotic number of inhomogeneous multigraphs
before and near the birth of complex components.
Then, in Section~\ref{sec:applications},
we show how to use this theorem to derive the phase transition
of several tractable satisfiability problems,
namely bipartitness, quantified $2$-Xor-formulas
and random~$k$-colorings.
Section~\ref{sec:proof} is dedicated to the proof of the main theorem:
first we derive the generating functions
of the vertex-colored and edge-weighted trees,
unicyclic graphs, paths of trees 
and graphs with cubic kernel of fixed excess.
They are then gathered to build general multigraphs.

    \section{Model and Main Theorem} \label{sec:model_theorem}

We consider labelled multigraphs
- loops and multiple edges are allowed -
with colored vertices and weighted edges.
Let~$\cm$ be a symmetric~$q \times q$ matrix
with non-negative coefficients
and~$\sigma$ a fixed positive constant.
Let~$\{c_1, \ldots, c_q\}$ be a set
of~$q$ distinct colors. A multigraph~$G$ is a
\emph{$\cms$-multigraph} if
\begin{itemize}
  \item each vertex~$v$ of~$G$ is colored with color $c(v) \in \{c_1, \ldots, c_q\}$,
  \item each edge $\overline{vw}$ of $G$ is weighted with~$\cm_{c(v), c(w)}$,
  \item a weight $\sigma$ is given to each connected component.
\end{itemize}
Following~\cite{B80, BF85, FKP89, JKLP93}, 
the \emph{compensation factor}~$\kappa(G)$ of a multigraph~$G$
with set of vertices $V$ and set of edges $E$ is
\begin{equation} \label{eq:COMPENSATION_FACTOR}
  \kappa(G) := \prod_{v \in V} 2^{-m_{v,v}} \prod_{\overline{vw} \in E} \left( m_{v,w}! \right)^{-1}
\end{equation}
where $m_{v,w}$ is the number of edges binding~$u$ to~$v$ in~$G$.
Let us consider an ordered sequence of $m$ ordered couples of vertices
$(v_1, w_1), \ldots, (v_m, w_m)$. Interpreting each couple as an edge,
this sequence describes a multigraph.
The number of such sequences corresponding to a given multigraph $G$
with $m$ edges is exactly $2^m m! \kappa(G)$.
Therefore, the two following random processes
induce the same probability distribution on the multigraphs with $n$ vertices and $m$ edges:
\begin{itemize}
\item
draw among all multigraphs with $n$ vertices and $m$ edges
with probability proportional to the compensation factors,
\item
draw uniformly and independently $2 m$ vertices in $[1,n]$ to form a sequence of couples
$(v_1, w_1), \ldots, (v_m, w_m)$, output the corresponding multigraph.
\end{itemize}
The compensation factor is equal to~$1$
if and only if the multigraph contains neither loops nor multiple edges
(such a multigraph is called \emph{simple}).
The sum of the compensation factors of all multigraphs
with~$n$ vertices and~$m$ edges is called, for simplicity, their \emph{number}
and is equal to~$\frac{n^{2m}}{2^m m!}$ (which needs not be an integer).
Given an $\cms$-multigraph $G$, we define its \emph{weight}~$\omega(G)$ 
as the product of three terms: 
the compensation factor~$\kappa(G)$,
a factor $\sigma$ for each connected component
and the product of the weights of its edges
\begin{equation} \label{def:WEIGHT-G}
  \omega(G) = \kappa(G) \times \sigma^{\cc(G)} 
  \times \prod_{\overline{vw} \in E} \cm_{c(v), c(w)} , 
\end{equation}
where $\cc(G)$ is the number of connected components of~$G$
and $E$ its set of edges.

There are three differences between our model and the original one,
introduced by S\"oderberg~\cite{S02}.
First, the number of edges is a parameter of the model,
while in~\cite{S02} each pair of vertices
is linked by an edge with some probability.
This is the same difference as between 
the classic graph models $G(n,m)$ and $G(n,p)$.
Secondly, we consider multigraphs instead of simple graphs.
Thirdly, the parameter $\sigma$ is new.
It brings to the model the expressiveness needed
to encode the constraint satisfaction problems
considered in Section~\ref{sec:applications}.

An edge-weighted multigraph is \emph{vertex-transitive}
if its automorphism group is transitive and 
also preserve the weights
-- see for instance Godsil and Royle~\cite{GR01}.
Intuitively, this means that, using only the topology of the multigraph,
no vertex can be distinguished from another.
Let~$G$ be a multigraph with~$q$ vertices and
weighted edges.
The \textit{weighted adjacency matrix} $\cm$ of~$G$ is
a~$q \times q$ matrix with entry~$\cm_{i,j}$
equal to the sum of the weights
of the edges between vertex~$i$ and vertex~$j$.
For simplicity, we say that a matrix~$\cm$
is \emph{vertex-transitive}
if it is symmetric, has non-negative coefficients 
and the weighted multigraph
associated is vertex-transitive.
The special structure of those matrices
implies several properties, in particular of their spectrum,
which are listed in Lemma~\ref{th:cm}.
%
%
Many models using~$\cms$-multigraphs
involve vertex-transitive matrices~$\cm$, 
e.g. the~$2$-colorability 
and the quantified~$2$-Xor satisfiability problems,
as will be shown in Section~\ref{sec:applications}.
Since our aim is to emphasize the link between
the birth of complex components
and the phase transition of some satisfiability or constraint problems,
we focus on the case where~$\cm$ is vertex-transitive.

We define $g_{\cm, \sigma}(n, m)$ 
as the sum of the weights of the $\cms$-multigraphs
built with $n$ vertices and $m$ edges
\[
  g_{\cm, \sigma}(n, m) = \sum_{|G| = n,\ \|G\| = m} \omega(G). 
\]

\begin{theorem}\label{th:main_vt}
  Let $\cm$ be a~$q \times q$ vertex-transitive matrix 
  with greatest eigenvalue~$\delta$,
  $\sigma$ a positive fixed constant,
  $c$ the number of connected components 
  in the multigraph associated to $\cm$
  and let $\chi(X)$ denote the polynomial~$\prod_{\lambda \in \spectre(\cm) \setminus \delta} 
	\left( 1- \frac{\lambda}{\delta } X \right)$, where $\spectre(\cm)$ is the spectrum of $\cm$.
  For any~$m/n$ in a closed interval of~$]0, 1/2[$,
\[
  g_{\cm, \sigma}(n, m) \sim
    \frac{n^{2m}}{2^m m!}
    \left(1-\frac{2m}{n}\right)^{\frac{1-c \sigma}{2}}
    \frac{\delta^m (\sigma q)^{n-m}}
      {\chi(\frac{2m}{n})^{\sigma/2}}.
\]
  As $n$ is large and $m={\frac{n}{2}}(1+\mu n^{-1/3})$ with~$|\mu| \leq n^{1/12}$, 
\[
  g_{\cm, \sigma}(n, m) \sim
  \frac{n^{2m} }{2^m m!}  
  \phi_{c \sigma}(\mu)
  n^{(c \sigma - 1)/6}  
  \frac{\delta^m (\sigma q)^{n-m}}
  { \chi(1)^{\sigma/2} } 
\]
  where~$\phi_{\sigma}(\mu)$ is equal 
  to~$\sqrt{2\pi} \sum_k e_k^{(\sigma)} \sigma^k A(3k+\frac{\sigma}{2}, \mu)$,
  $e_k^{(\sigma)}$ is the~$(2k)$-th coefficient
  of~$\left( \sum_{n \geq 0} \frac{(6n)! z^{2n}}{(2n)! (3n)! 2^n (3!)^n} \right)^\sigma$
  and~$A(y,\mu) = \frac{e^{-\mu^3/6}}{3^{(y+1)/3}}
    \sum_{k \geq 0} \frac{(3^{2/3} \mu /2)^k}{k! \Gamma((y+1-2k)/3)}$.
\end{theorem}
\noindent \textbf{Remarks.} 
\begin{enumerate}
\item
As a corollary of Lemma~\ref{th:cm}, 
the number~$\cgcm$ of connected components in the graph with adjacency matrix~$\cm$
is equal to the multiplicity of the greatest eigenvalue~$\delta$ of this vertex-transitive matrix.
\item \label{item:chi_and_P}
The polynomial
\[
  \chi(X) = \prod_{\lambda \in \spectre(\cm) \setminus \delta} 
  \left( 1- \frac{\lambda}{\delta } X \right)
\] 
is linked to the characteristic polynomial~$P_\cm(X) = \det(X \Id - \cm)$ by the relation
\[
  \chi(X) = \left(\frac{X}{\delta}\right)^q \frac{P_\cm(\frac{\delta}{X})}{(1-X)^c}.
\]
\item
Since $\cm$ is a symmetric matrix with dominant eigenvalue $\delta$,
all the values in $\spectre(\cm) \setminus \delta$ are real and smaller that $\delta$.
Therefore, $\chi(1)$ is positive.
\end{enumerate}

    \section{Applications} \label{sec:applications}

To describe the phase transition of a problem,
we perform the following steps:
\begin{enumerate}
 \item build $\cm$ and $\sigma$ in order to obtain
  a one-to-one mapping between the instances of the problem and the $\cms$-multigraphs,
  \item derive from~$\cm$ the values~$q$, $\cgcm$, $\delta$, $\chi(\frac{2m}{n})$ and $\chi(1)$ 
  defined in Theorem~\ref{th:main_vt},
  \item apply Theorem~\ref{th:main_vt}.
\end{enumerate}
In the following section, we rediscover some results 
from Pittel and Yeum~\cite{PY10},
prove new results on the satisfiability 
of quantified~$2$-Xor-formulas~\cite{CDE07}
and rederive the probability that a random~$k$-coloring is proper~\cite{W72}.

    \subsection{Bipartite Multigraphs}

A \emph{proper} 2-coloring of a multigraph
is a way of coloring the vertices with~$2$ colors
such that no two adjacent vertices share the same one.
A graph is \emph{bipartite} if it admits a proper $2$-coloring.
In particular, such a graph contains no loop.
In~\cite{PY10}, the authors computed
the probability for a random graph with~$n$ vertices and~$m$ edges to be bipartite,
and we propose a new proof of some of their results.

Let~$G$ be a multigraph with~$n$ vertices
and~$c$ a function from~$[1,n]$ to~$\{1,2\}$.
We define the vertex-colored and edge-weighted
multigraph~$G_c$ as follows:
a color~$c(v)$ is assigned to each vertex~$v$,
each edge~$\overline{vw}$ has weight~$1$
if~$c(v) \neq c(w)$ and~$0$ if~$c(v) = c(w)$.
The weight~$\omega(G_c)$ of~$G_c$ is~$\kappa(G)$ 
times the product of the weights of the edges.
Therefore, $\omega(G_c) = 0$ if and only if
there exist adjacent vertices~$v$ and~$w$
such that~$c(v) = c(w)$.
It follows that
$\sum_{c:[1,n] \to \{1,2\}} \omega(G_c)$
is the number of ways to properly $2$-color~$G$.
We just described a one-to-one mapping between
the properly $2$-colored multigraphs
and the $\left(\left(\begin{smallmatrix}0&1\\1&0\end{smallmatrix}\right),1\right)$-multigraphs.

Every bipartite multigraph~$G$ admits~$2^{\cc(G)}$ proper $2$-colorings,
because such a coloring is characterized
by the choice of one color in each connected component. 
Therefore, to count each bipartite multigraph exactly one time, 
each connected component receives a compensation factor~$\frac{1}{2}$. 
This proves that the bipartite multigraphs are in a one-to-one mapping 
with the~$\left(\left(\begin{smallmatrix}0&1\\1&0\end{smallmatrix}\right),\frac{1}{2}\right)$-multigraphs.

For $\cm = \left(\begin{smallmatrix}0&1\\1&0\end{smallmatrix}\right)$,
we have $\delta = 1$, $\chi(X) = 1+X$, $q = 2$, $\cgcm = 1$ and~$\sigma = \frac{1}{2}$.
As a corollary of Theorem~\ref{th:main_vt}, we thus have

\begin{theorem}
  Let $\probabip(n, m)$ denote the probability that
  a random graph or multigraph
  with~$n$ vertices and~$m$ edges is bipartite, then
  \begin{itemize}
  \item when~$m/n$ is in a closed interval of~$]0,1/2[$,
  \[
  \probabip(n, m) \sim \left( \frac{1 - \frac{2m}{n}}{1+\frac{2m}{n}} \right)^{1/4},
  \]
  \item for any~$|\mu| \leq n^{1/12}$, 
  \[
   \lim_{n \rightarrow \infty} n^{1/12}\ \probabip\Bigl (n,\frac{n}{2}(1+\mu n^{-1/3})\Bigr ) = \phi_{1/2}(\mu),
  \]
  \end{itemize}
  where~$\phi_{1/2}(\mu)$, defined in Theorem~\ref{th:main_vt},
  decreases from~$1$ to~$0$ for~$\mu$ in~$\RR$.
\end{theorem}


    \subsection{Quantified $2$-Xor Formulas}

In~\cite{CDE07}, the authors analyze
quantified~$2$-Xor formulas.
Those are quantified conjunctions of $m$~Xor-clauses 
with $\beta$~universal and $n$~existential variables
\begin{equation} \label{eq:2QXorSAT}
 \forall x_1\hdots x_{\beta} \exists y_1 \hdots y_n 
  \bigwedge_{i=1}^m \left( 
    y_{f_{i,1}} \oplus y_{f_{i,2}} 
    = (e_{i,1} \wedge x_1) 
    \oplus \cdots
    \oplus (e_{i, \beta} \wedge x_\beta)
  \right).
\end{equation}
The values of the variables $(x_i)$ and $(y_j)$ 
can be considered equally as Booleans or bits, 
by identifying \emph{True} to $1$ and \emph{False} to $0$.
The Boolean operator \emph{Xor} $x \oplus y$ corresponds to
the bit sum $(x + y \mod 2)$, 
and the \emph{And} Boolean operator $x \wedge y$ 
to the product $(x y \mod 2)$.
The authors study how the probability
of satisfiability evolves
with the number~$m$ of clauses when
the number~$n$ of existential variables is large,
and locate the value of~$m$ at which
the phase transition occurs.

Each clause
\[
    y_{f_{i,1}} \oplus y_{f_{i,2}} 
    = (e_{i,1} \wedge x_1) 
    \oplus \cdots
    \oplus (e_{i, \beta} \wedge x_\beta)
\]
is characterized by a triplet $(f_{i,1}, f_{i,2}, e_i)$
where $f_{i,1}$ and $f_{i,2}$ are integers in $[1,n]$
and $e_i$ is a $\beta$-tuple of bits.
We consider clauses such that 
$e_i$ belongs to a fixed multiset~$E$
of~$\beta$-tuples of bits.
We call the formulas that contain only
those clauses the \emph{$E$-formulas}.
For example, $y_1 \oplus y_2 = x_1 \oplus x_3$
is a $\{(1,0,1), (1,1,0)\}$-formula (with only one clause),
but $y_1 \oplus y_2 = x_2 \oplus x_3$ is not.
For any integer $j$ in $[1, 2^\beta]$, 
$\bin{j}$ is the $\beta$-tuple of bits 
matching the binary decomposition of~$j-1$
\[
  \bin{j} = (b_0, \ldots, b_{\beta-1})
  \text{ \ if and only if \ }
  j-1 = \sum_{k=0}^{\beta-1} b_k 2^k.
\]
To a multiset~$E$ of~$\beta$-tuples of bits, 
we associate a matrix~$\cm^{(E)}$ 
of dimension~$2^{\beta} \times 2^{\beta}$ 
such that~$\cm^{(E)}_{i,j}$ is the number of occurrences
of~$\bin{i} \oplus \bin{j}$ in~$E$:
\[
 \cm^{(E)}_{i,j} = \#\{ e \in E\ |\ \bin{i} \oplus \bin{j} = e\}.
\]
For example, when~$\beta = 2$ and~$E = \{\smat{0&1}, \smat{1&0}\}$, 
we have~$R^{(E)} = \smat{0&1&1&0\\ 1&0&0&1 \\ 1&0&0&1 \\ 0&1&1&0}$.

\begin{lemma} \label{th:bij2qxor}
  Let~$E$ be a multiset of $\beta$-tuples of bits.
  There exists a one-to-one mapping between
  \begin{itemize}
  \item the satisfiable $E$-formulas with~$n$ existential variables
    and~$m$ clauses,
  \item the $(R^{(E)},2^{-\beta})$-multigraphs with $n$~vertices and $m$~edges.
  \end{itemize}
\end{lemma}
\begin{proof}
  Let $\phi$ denote the formula of Equation~\eqref{eq:2QXorSAT} 
  and assume it is an $E$-formula.
  A \emph{solution} of $\phi$ is a set $\eta_1, \ldots, \eta_n$ of $n$ $\beta$-tuples of bits
  such that for each instantiation of the variables~$x_1, \ldots, x_\beta$,
  the values
  \begin{align*}
    y_1 &= \eta_{1,1} x_1 \oplus \cdots \oplus \eta_{1,\beta} x_\beta, \\
    &\qquad \vdots \\
    y_n &= \eta_{n,1} x_1 \oplus \cdots \oplus \eta_{n,\beta} x_\beta
  \end{align*}
  satisfy~$\phi$.
  For example, the formula~$\forall x_1, x_2\ \exists y_1, y_2,\ y_1 \oplus y_2 = x_1$
  has~$4$ solutions 
  $\{\smat{1&0}, \smat{0&0}\}$,
    $\{\smat{1&1}, \smat{0&1}\}$,
    $\{\smat{0&0}, \smat{1&0}\}$,
  and~$\{\smat{0&1}, \smat{1&1}\}$.
  The first one matches the obvious solution~$y_1 = x_1$, $y_2 = 0$.

  We first build a bijection
  between the couples ($E$-formula, solution) 
  and the $(\cm^{(E)},1)$-multigraphs.
  Each existential variable~$y_i$
  matches a vertex of color $\eta_i$,
  and each clause an edge.
  The number of couples ($E$-formula, solution) is
  \[
    \sum_{\phi \in \text{$E$-formula}} \ \sum_{\text{solution of~$\phi$}} 1.
  \]
  The proof consists in switching the sums,
  assigning to each~$y_i$ a linear combination
  \[
    y_i = \eta_{i,1} x_1 \oplus \cdots \oplus \eta_{i,\beta} x_\beta,
  \]
  and to count the number of~$E$-formulas satisfied by those~$(y_i)$.
  By definition of~$\cm^{(E)}$, this is equal to~$g_{\cm^{(E)}, 1}(n,m)$.
  %

  To end the proof, we show that
  a satisfiable~$E$-formula admits~$2^{\beta \cc}$
  solutions, where~$\cc$ is the number of connected
  components in its graph representation.
  Indeed, once the color (i.e. the~$\beta$-tuple)
  of an existential variable is chosen,
  a transversal of the graph determines
  the colors of the other existential variables of the component.
  So we have exactly one choice of color for each connected component,
  and~$2^\beta$ choices for this color.
\end{proof}

This application is an opportunity
to present some tools for deriving
the parameters $c$ and~$\chi(X)$
of Theorem~\ref{th:main_vt}
for non-trivial matrices~$\cm$.
Let~$\ham(\beta)$ denote the $2^{\beta}\times 2^\beta$ matrix
defined by
\[
 \ham(\beta)_{i,j} = \begin{cases}
  1 & \text{if the Hamming distance between~$\bin{i}$ and~$\bin{j}$ is~$1$,}\\
  0 & \text{otherwise.}
  \end{cases}
\]

    \subparagraph*{Xor-Clauses with one Universal Variable.}

If the Xor-clauses contain exactly one universal variable, like
$ y_i \oplus y_j = x_k$,
then~$E$ is the set~$\{e_1, \ldots, e_{\beta}\}$, 
where, for all~$i$, $e_i$ denotes the $\beta$-tuple of bits with a~$1$ at position~$i$ and~$0$ elsewhere,
and~$R^{(E)} = \ham(\beta)$. In this case, the number of colors is~$q = \sigma^{-1} = 2^{\beta}$.
The matrix~$\ham(\beta)$ is irreducible, so~$c = 1$.
Its greatest eigenvalue $\delta = \beta$ corresponds to the eigenvector~$\vect{1} = \transp{\smat{1 & \cdots & 1}}$.
The matrix~$\ham(\beta)$ admits the following block decomposition\footnote{We thank Timo Jolivet who helped us find this recursion.}:
\[
 \ham(\beta+1) = \left(\begin{smallmatrix} \ham(\beta) & \Id\\ \Id & \ham(\beta) \end{smallmatrix}\right),
\]
so its characteristic polynomial is solution of the following recursive formula:
\[ 
  P_{\ham(\beta+1)}(X) = \det((X \Id - \ham(\beta))^2 - \Id^2) 
  = P_{\ham(\beta)}(X-1) P_{\ham(\beta)}(X+1).
\]
By induction, 
\[
 P_{\ham(\beta)}(X) = \prod_{i=0}^\beta (X - (\beta - 2i))^{\binom{\beta}{i}}
\]
and, using Remark~\ref{item:chi_and_P},
\[
  \chi(X) = \prod_{i=1}^\beta \left(1 - X \left(1-\frac{2i}{\beta}\right)\right)^{\binom{\beta}{i}}.
\]

    \subparagraph*{Xor-Clauses with $\alpha$ Universal Variables.}

We consider Xor-clauses that contain 
the ordered sum of exactly~$\alpha$ universal variables,
e.g. for~$\alpha = 3$, $y_1 \oplus y_2 = x_1 \oplus x_3 \oplus x_1$.
In this case, $R^{(E)} = \ham(\beta)^\alpha$.
The parameters of this matrix are derived from those of~$\ham(\beta)$.
The size $q = 2^\beta$ is the same.
The eigenvalues of~$\ham(\beta)^\alpha$ are those of~$\ham(\beta)$, raised at the power~$\alpha$.
In particular, the greatest eigenvalue of~$\ham(\beta)^\alpha$ is~$\delta = \beta^\alpha$.
If~$\alpha$ is odd, $\ham(\beta)^\alpha$ is irreducible and~$c=1$.
If~$\alpha$ is even, the greatest eigenvalue of~$\ham(\beta)^\alpha$ has multiplicity~$2$,
so~$c = 2$.
Finally, 
\[
  \chi(X) = \prod_{i=1}^{\beta+1-c} 
  \left(1 - X \left(1-\frac{2i}{\beta}\right)^\alpha \right)^{\binom{\beta}{i}}.
\]

    \subparagraph*{Xor-Clauses with Distinct Universal Variables.}

If, furthermore, the~$\alpha$ universal variables in each Xor-clause
are constrained to be distinct, then 
$E$ is the set $\{ e \in \{0,1\}^\beta\ |\ e_1 + \cdots + e_\beta = \alpha \}$.
Let $\ham(\alpha,\beta)$ denote the matrix~$R^{(E)}$.
We claim that this matrix satisfies the recursive relation for all~$0 \leq \alpha \leq \beta -2$
\[
  \ham(\beta)\ham(\alpha+1,\beta) = (\beta-\alpha)\ham(\alpha,\beta)+(\alpha+2)\ham(\alpha+2,\beta).
\]
Indeed, the coefficient~$(i,j)$ of the matrix~$\ham(\beta)\ham(\alpha+1,\beta)$
is the number of ways to write the bit-to-bit sum
\begin{equation} \label{eq:qxorsat_distinct}
  \bin{i} \oplus \bin{j} = e \oplus v_{\alpha+1}
\end{equation}
where $e$ and $v_{\alpha+1}$ are $\beta$-tuples of bits
that contains respectively exactly~$\alpha+1$ and $1$ ones.
There are now three cases.
\begin{itemize}
\item
  If~$\bin{i} \oplus \bin{j}$ contains~$\alpha$ ones,
  then~$e$ has canceled a bit from~$v_{\alpha+1}$.
  There are then~$(\beta-\alpha)$ couples~$(v_{\alpha+1}, i)$
  which are solutions of Equation~\eqref{eq:qxorsat_distinct}.
\item
  If~$\bin{i} \oplus \bin{j}$ contains~$\alpha+2$ ones,
  then the one in~$e$ is added to a zero of~$v_{\alpha+1}$,
  and Equation~\eqref{eq:qxorsat_distinct}
  admits~$\alpha+2$ solutions.
\item
  Otherwise, Equation~\eqref{eq:qxorsat_distinct}
  has no solution.
\end{itemize}
Thus we can write
\[
  R^{(E)} = P_{\alpha,\beta}(\ham(\beta))
\]
where the polynomial~$P_{\alpha,\beta}$ 
is characterized by the recursive formula $P_{0,\beta}(X) = 1$,
$P_{1,\beta}(X) = X$ and for all $\alpha$ in $[0,\beta-2]$, then
\[
 (\alpha + 2) P_{\alpha+2,\beta}(X) = X P_{\alpha+1,\beta}(X) - (\beta - \alpha) P_{\alpha,\beta}(X).
\]
Again, the parameters of~$R^{(E)}$ are derived from those of~$\ham(\beta)$.
Observe that~$P_{\alpha,\beta}(\beta) = \binom{\beta}{\alpha}$ 
and~$P_{\alpha,\beta}(-\beta) = (-1)^\alpha \binom{\beta}{\alpha}$.
The size is $q = 2^\beta$, and
the greatest eigenvalue is $\delta = P_{\alpha,\beta}(\beta) = \binom{\beta}{\alpha}$.
If~$\alpha$ is odd, $\ham(\beta)^\alpha$ is irreducible and~$c=1$,
else~$c = 2$.
Finally
\[
 \chi(X) = \prod_{i=1}^{\beta+1-c} 
\left(1 - X \frac{P_{\alpha,\beta}\left(\beta-2i\right)}{\binom{\beta}{\alpha}} \right)^{\binom{\beta}{i}}.
\]

    \subparagraph*{Xor-Clauses with $\alpha$ Universal Variables and a Constant Term.}

We consider Xor-clauses of the form
\[
  y_i \oplus y_j = e_1 x_1 \oplus \cdots \oplus e_\beta x_\beta \oplus e_{\beta+1},
\]
where exactly $\alpha$ of the bits $e_1, \ldots, e_\beta$ are 1's.
The set~$E$ contains now~$(\beta+1)$-tuples of bit.
To take the term $e_{\beta+1}$ into account,
any solution of such a formula now assigns to each existential variable~$y_i$
a affine combination of universal variables plus a constant~$0$ or~$1$, so
\[
  y_i = \eta_{i,1} x_1 \oplus \cdots \oplus \eta_{i,\beta} x_\beta \oplus \eta_{i, \beta+1}.
\]
where~$\eta_i$ is a~$(\beta+1)$-tuple of bits.
Let~$E_\alpha$ denote the set~$E$ corresponding to Xor-clauses
with~$\alpha$ universal variables,
and~$E_{\alpha, \epsilon}$
the corresponding set with the option of adding a constant.
Each~$\beta$-tuple~$e$ in~$E_\alpha$ matches
two~$(\beta+1)$-tuples in~$E_{\alpha, \epsilon}$:
both with the same first~$\beta$ bits as~$e$,
one with last bit~$0$, and the other with last bit~$1$.
Therefore, the matrix~$R^{(E_{\alpha, \epsilon})}$ is equal to
\[
  \smat{ R^{(E_{\alpha})} & R^{(E_{\alpha})}\\ R^{(E_{\alpha})}& R^{(E_{\alpha})} }.
\]
The spectrum of $R^{(E_{\alpha, \epsilon})}$ is the same
as the one of $R^{(E_{\alpha})}$,
except that each eigenvalue is doubled and the eigenvalue~$0$ 
is added with multiplicity~$2^\beta$.
We then obtain the parameters
$q = 1/\sigma = 2^{\beta + 1}$,
$c  = 1$ if~$\alpha$ is odd and~$c = 0$ otherwise,
$\delta = 2 \beta^\alpha$
and~$\chi(X)$ as in the following theorem.
Injecting those parameters into Theorem~\ref{th:main_vt}
and dividing by the total number of~$E_{\alpha, \epsilon}$-formulas
\[
  \frac{n^{2m}}{2^m m!} (2 \beta^\alpha)^m
\]
gives the following result.

\begin{theorem} \label{th:qxorsat}
  Let us consider a random quantified $2$-Xor formula of the form~\eqref{eq:2QXorSAT}
  with $n$~existential variables,
  $\beta$~universal variables
  and $m$~Xor-clauses 
  containing two existential variables, 
  $\alpha$ universal variables and one constant term in~$\{0,1\}$.
  Let~$\psat$ denote the probability that such a formula is satisfiable,
  $c = 2$ if~$\alpha$ is even, $c = 1$ otherwise, and
  \[
   \chi(X) = \prod_{i=1}^{\beta+1-c} \left(1- 2 \left( 1 - X 2 i / \beta  \right)^\alpha \right)^{\binom{\beta}{i}},
  \] 
  then
  \begin{itemize}
  \item as $n$ is large and~$m/n$ is restrained to a closed interval of~$]0,1/2[$,
  \[
  \psat \sim \chi\left(2m/n\right)^{-2^{-\beta-2}} 
      \left( 1- 2m/n \right)^{\frac{1}{2} - c 2^{-\beta-2}}
  \]  
  \item for any fixed real value~$x$
  and~$m = \frac{n}{2}(1+\mu n^{-1/3})$,
  \[
   \lim_{n \rightarrow \infty} n^{(1-c 2^{-\beta-1})/6} \psat = \chi(1)^{-2^{-\beta - 2}} \phi_{c 2^{-\beta-1}}(\mu),
  \]
  \end{itemize}
  where~$\phi_\sigma(\mu)$ is a computable function defined in Theorem~\ref{th:main_vt}.
\end{theorem}

Combining the previous results, 
we could as well consider quantified~$2$-Xor formulas
with~$\alpha$ distinct universal variables and a constant term
in each clause.

    \subsection{Random $k$-Coloring of Random Multigraphs}

The following theorem, due to Wright~\cite{W72}, enumerates
the properly $q$-colored multigraphs.
We propose a new proof using the formalism of $\cms$-multigraphs.

\begin{theorem} \label{th:kcolor}
  If $m/n$ is fixed in $]0,1/2[$,
  the asymptotic probability that a random $q$-coloring 
  of a random multigraph with~$n$ vertices and $m$ edges
  is proper is
  \[
    \left( 1 - \frac{1}{q} \right)^m 
    \left( 1 + \frac{1}{q-1} \frac{2 m}{n} \right)^{-\frac{q-1}{2}}.
  \]  
\end{theorem}

\begin{proof}
A multigraph properly~$q$-colored
is a~$(\cm,1)$-multigraph where~$\cm_{i,j} = 1$ for all $i \neq j$ and $0$ otherwise.
Their asymptotics is derived from Theorem~\ref{th:main_vt}
with the parameters~$c=1$, $\chi(X) = \left( 1 + X/(q-1) \right)^{q-1}$.
It is then divided by the total number of multigraphs 
with~$n$ vertices and~$m$ edges 
randomly (and possibly not properly) $q$-colored,
which is~$\frac{n^{2m}}{2^m m!} q^n$.
\end{proof}
In fact, this result holds for any positive fixed value of~$m/n$~\cite{W72}.
There is a proof of this result in the setting of inhomogeneous multigraphs,
which is an interesting development that will be part of a forthcoming publication.
Here, we just sketch this proof.
It starts with the direct expression
\[
  g_{\cm, 1}(n,m)
=
  \frac{1}{2^m m!} \sum_{\substack{\vect{n} \in \NN^q \\ \gvect{1} \vect{n} = n}}
  \binom{n}{n_1, \ldots, n_q}
  \left( \gvect{n} \cm \vect{n} \right)^m.
\]
We then apply Theorem 5.4.8 of~\cite{PW13} to conclude.
This approach can be generalized to any irreducible aperiodic matrix~$\cm$
when~$\sigma = 1$.

Theorem~\ref{th:kcolor} is not to be confused with
an asymptotic of~$q$-colorable multigraphs,
because a colorable multigraph may have
several proper colorings.

    \section{Proof of Theorem~\ref{th:main_vt}} \label{sec:proof}

    \subsection{Properties of Vertex-Transitive Matrices}

The notation~$\vect{1}$ stands for the column vector
with all coefficients equal to~$1$.
The structure of a vertex-transitive matrix implies the following properties:
\begin{lemma} \label{th:cm}
  Let $\cm$ denote a~$q \times q$ vertex-transitive matrix
  with non-negative coefficients, then there exist
 \begin{enumerate}
  \item an integer~$\cgcm$, a permutation matrix~$P$ 
    and a square matrix~$\ccm$ of dimension~$\frac{q}{\cgcm} \times \frac{q}{\cgcm}$ such that
    $P \cm P^{-1} = \diag{\ccm, \hdots, \ccm}$,
    where the block diagonal matrix contains $\cgcm$~blocks,
  \item a positive~$\delta$, eigenvalue of~$S$ of multiplicity~$1$, such that
    $\ccm \vect{1} = \delta \vect{1}$
    and for all~$\lambda$ in the spectrum of $\ccm$, $|\lambda| \leq \delta$,
  \item an orthogonal matrix~$Q$
    and a diagonal matrix~$\Delta$ such that $\ccm = Q \Delta \transp{Q}$,
    $\Delta_{1,1} = \delta$ and $Q \vect{e_1} = \sqrt{\frac{c}{q}} \vect{1}$.
 \end{enumerate}
\end{lemma}
\begin{proof}
Because the multigraph~$G$ associated to~$\cm$
is vertex-transitive, 
all pairs of connected components are isomorphic.
The matrix~$S$ denotes the weighted adjacency matrix
of one of those components and~$c$ is their number.
In an edge-weighted multigraph,
the \emph{degree} of a vertex $v$ is the sum
of the weights of the edges that contain $v$.
All the degrees in~$G$ are equal, otherwise
two vertices with different degrees could be distinguish,
so~$\vect{1}$ is an eigenvector of~$\cm$ 
and~$\delta$ denotes this common degree.
Applying the Perron-Fr\"obenius Theorem,
we conclude that the eigenvalue~$\delta$
has multiplicity~$1$
and is greater or equal in absolute value than
any other eigenvalue.
Finally, real symmetric matrices are diagonalizable by orthogonal matrices.
\end{proof}

Let~$\delta$ denote the greatest real eigenvalue of~$\cm$.
We can assume it to be equal to~$1$
without loss of generality,
replacing~$\cm$ by~$\frac{1}{\delta}\cm$
and~$g(n,m)$ by~$\delta^m g(n,m)$.
We can also assume that the number~$c$ 
of connected components of the multigraph encoded 
by~$\cm$ is~$1$:
the $(\cm,\sigma)$-multigraphs 
are in a one-to-one mapping with
the $(S, c \sigma)$-multigraphs 
where~$S$ is the adjacency matrix of one 
of the connected components.
In the rest of this section,
$\cm$ is assumed to be 
a~$q \times q$ irreducible vertex-transitive matrix
with greatest eigenvalue~$1$.

    \subsection{Trees and Unicyclic Components}

In~\cite{JKLP93}, graphs are decomposed in three parts:
trees, unicyclic components and \emph{complex} components~\cite{W77}.
Their generating functions are expressed in term of the Cayley tree function~$T(z)$
that counts the rooted labelled trees and is characterized by the equation
$ T(z) = z e^{T(z)}$.
We follow the same approach.

An \emph{$\cm$-tree} is a connected $\cm$-multigraph without cycle.
If one vertex is marked, we say that the tree is \emph{rooted}.
A connected $\cm$-multigraph with exactly one cycle
is called an \emph{$\cm$-unicyclic multigraph}.
Let $T_i(z)$, $U(z)$ and $V(z)$ denote the generating functions
of $\cm$-rooted trees with root of color~$i$,
unrooted trees and unicyclic multigraphs.
Let also $\vect{T}(z)$ denote the vector
$\transp{\left(\begin{smallmatrix} T_1(z) & \cdots & T_q(z) \end{smallmatrix}\right)}$.
A \emph{$\cm$-path of trees} is a colored directed path
that links two vertices (that may not be distinct)
of color~$i$ and~$j$, and
each internal vertex of the path is the root
of a colored $\cm$-rooted tree.
Its generating function is denoted by~$P_{i,j}(z)$.

\begin{lemma} \label{th:TUVPath}
  If $\cm$ is irreducible with greatest eigenvalue $1$,
  the generating functions of $\cm$-rooted trees, unrooted trees, unicyclic graphs
  and paths of trees are
  \begin{flalign*}
    &\textstyle \vect{T}(z) = T(z) \vect{1}
    &\textstyle  V(z) = -\frac{1}{2} \log(1-T(z)) -\frac{1}{2} \log( \chi(T(z)) ) \\
    &\textstyle U(z) = q ( T(z) - \frac{1}{2} T(z)^2) 
    &\textstyle P_{i,j}(z) = \frac{1}{q(1-T(z))} + 
      \sum_{l=2}^q Q_{i,l} Q_{j,l} \frac{\Delta_{l,l}}{1 - \Delta_{l,l} T(z)}
  \end{flalign*}
  where $T(z)$ is the Cayley tree function
  and $\cm = Q \Delta \transp{Q}$ as in Lemma~\ref{th:cm}.
\end{lemma}
\begin{proof}
  Using the analytic combinatorics tools (a good reference is~\cite{FS09}), 
  the combinatorial specification of $\cm$-rooted trees translates into
  the following equations: for all~$i$,
  $T_i(z)$ is equal to~$z \exp(\row_i(\cm) \vect{T}(z))$.
  Since $\cm$ is vertex-transitive, for all~$i,j$, $T_i(z) = T_j(z)$, so
  $T_i(z) = \frac{1}{\delta} T(\delta z)$.
  an $\cm$-unrooted tree with a marked vertex is
  an $\cm$-rooted tree with root of unknown color, so
  $z U'(z) = \sum_{i=1}^q T_i(z)$.
  Similarly, an $\cm$-unicyclic graph with a marked vertex on its cycle
  and an orientation
  is an $\cm$-path of rooted trees, so
  $u \partial_u V(z,u) = \frac{1}{2} \sum_{i=1}^q \sum_{k \geq 1} (u T(z) \cm)^k_{i,i}$
  where $u$ marks the vertices of the cycle and $V(z) = V(z,1)$.
  Finally, $\PATH_{i,j}(z) = (\cm (\Id - T(z) \cm)^{-1})_{i,j}$
  and Lemma~\ref{th:cm} lead to the announced expression.
\end{proof}

Observe that at the first order, $U(z)$, $V(z)$ and $\PATH_{i,j}(z)$ 
are equal or proportional to their non-colored counterparts
$T(z) - \frac{1}{2} T(z)^2$, 
$-\frac{1}{2} \log(1-T(z))$
and $\frac{1}{1-T(z)}$.
Furthermore, the first order of~$\PATH_{i,j}(z)$ 
is independent of~$i$ and~$j$.

We will prove in Theorem~\ref{th:complexpart1}
that when~$m/n < 1/2$,
almost all $(\cm,\sigma)$-multigraphs
with~$n$ vertices and~$m$ edges
contain only trees and unicyclic components.
Theorem~\ref{th:subcritical} is then equivalent to
the first statement of Theorem~\ref{th:main_vt}.
\begin{theorem} \label{th:subcritical}
  With the notations of Theorem~\ref{th:main_vt},
  the number of $(\cm,\sigma)$-multigraphs 
  that contain only trees and unicyclic components is
  \[
    g^{(0)}_{\cm,\sigma}(n,m)
    \sim
    \frac{n^{2m}}{2^m m!} 
    \left( 1 - \frac{2m}{n}\right)^{\frac{1-\sigma}{2}}
    \frac{(q \sigma)^{n-m}}{\chi\left(\frac{2m}{n}\right)^{\sigma/2}}.
  \]
\end{theorem}
\begin{proof}
A multigraph without complex component
is a set of~$n-m$ trees and of unicyclic components, so
\begin{equation} \label{eq:subcritical}
  g^{(0)}_{\cm,\sigma}(n,m)
  =
  n! [z^n] \frac{(\sigma U(z))^{n-m}}{(n-m)!} e^{\sigma V(z)}.
\end{equation}
We then apply Theorem VIII.8 of~\cite[p.587]{FS09}
to derive the asymptotics of the coefficient extraction.
\end{proof}

    \subsection{Complex Components}

The notions of \emph{excess} and \emph{kernel}
were first combined with a generating function approach in~\cite{W77}
and then~\cite{JKLP93}. This section relies on their work.
The \emph{excess} of a graph is defined as the difference between 
the number of edges and of vertices~$k = m-n$.
A component with excess~$(-1)$ (resp.~$0$) is a tree (resp. unicyclic).
The \emph{complex} part of a multigraph
is the set of its connected components that have
positive excess.
Deleting the vertices of degree one 
and merging the vertices of degree two,
each graph can be reduced to a simpler graph,
called its \emph{kernel}, 
with minimum degree at least three.
Reciprocally, any such graph can be developed by
replacing edges by paths and adding trees to the vertices.
The set~$\kernels_k$ of kernels of excess~$k$ is finite.
Among them, the kernels that maximize the number of edges
are the cubic (i.e. $3$-regular) multigraphs~$\cubics_k$. 
with~$2 k$ vertices and~$3 k$~edges.
Their number, counted
with their compensation factors and a weight~$\sigma$ 
for each connected component, is computable
\[ 
  |\cubics_{k,\sigma}| = \sum_{G \in \cubics_k} \kappa(G) \sigma^{\cc(G)} 
  = (2k)! [z^{2k}]\left(\sum_{n\geq 0} \frac{(6n)!}{(3!)^{2n} 2^{3n} (3n)!} \frac{z^{2n}}{(2n)!}\right)^\sigma. 
\]
The generating function of complex (i.e.~without trees and unicyclic components) 
$(\cm,\sigma)$-multigraphs of excess~$k$ is
\begin{equation} \label{eq:K_k_exact}
  K_{k,\sigma}(z) = 
    \sum_{G \in \kernels_k}
    \sum_{\vect{c} \in [1,q]^{|G|}}
      \frac{\kappa(G) \sigma^{\cc(G)}}{|G|!}
      \prod_{i \in [1,|G|]} T_{c_i}(z) 
      \prod_{(i,j) \in \operatorname{edge}(G)} \PATH_{c_i,c_j}(z). 
\end{equation}
Since~$\kernels_k$ is finite, this generating function
is a rational function in~$T(z)$.
In its partial fraction decomposition,
the term with denominator containing the highest power of~$1-T(z)$
is
\begin{equation} \label{eq:Kcubic}
  \frac{|\cubics_{k,\sigma}|}{(2k)! q^k} \frac{T(z)^{2k}}{(1-T(z))^{3k}}.
\end{equation}

\begin{theorem} \label{th:complexpart1}
When~$m/n < 1/2$ is fixed, almost all $(\cm,\sigma)$-multigraphs
have an empty complex part.
\end{theorem}
\begin{proof}
In all the proof, $m/n < 1/2$ is assumed to be fixed.
A multigraph with~$n$ vertices, $m$ edges
and complex part of excess~$k$ is
a set of~$n-m+k$ trees, a set of unicyclic components
and a complex part. Therefore, the number of such~$(R,\sigma)$-multigraphs is
\begin{equation} \label{eq:gk}
  g^{(k)}_{\cm,\sigma}(n,m) 
  = 
  n! [z^n] \frac{(\sigma U(z))^{n-m+k}}{(n-m+k)!} e^{\sigma V(z)} K_{k,\sigma}(z).
\end{equation}
With~$g^{(0)}_{\cm,\sigma}(n,m)$ defined as in Theorem~\ref{th:subcritical},
the theorem states that
when~$m/n < 1/2$ is fixed,
\[
  g_{\cm,\sigma}(n,m) 
  \sim g^{(0)}_{\cm,\sigma}(n,m).
\]
Since
\[
  g_{\cm,\sigma}(n,m) 
  =
  \sum_{k \geq 0} g^{(k)}_{\cm,\sigma}(n,m),
\]
this is equivalent with
\begin{equation} \label{eq:complexpart1sum}
  \lim_{n \rightarrow \infty}
  \sum_{k \geq 1} \frac{g^{(k)}_{\cm,\sigma}(n,m) }{g^{(0)}_{\cm,\sigma}(n,m) }
  = 0.
\end{equation}

It is well known that Theorem~\ref{th:complexpart1} holds for classic multigraphs.
Therefore, if $g^{(k)}_{(1),1}(n,m)$ denotes the sum
of the compensation factors of multigraphs
with $n$ vertices, $m$ edges and complex part of excess $k$,
then 
$g^{(0)}_{(1),1}(n,m)$ has the same asymptotics 
as the total number of multigraphs
\[
  g^{(0)}_{(1),1}(n,m) \sim \frac{n^{2m}}{2^m m!}.
\]
Combined with Theorem~\ref{th:subcritical}, this equivalence implies
that there exists a constant $C_1$, which depends only on $m/n$,
such that for $n$ large enough,
\begin{equation} \label{eq:minor}
  g^{(0)}_{R,\sigma}(n,m) \geq C_1 (q \sigma)^{n-m} g^{(0)}_{(1),1}(n,m).
\end{equation}
Since an $\cms$-multigraph is a multigraph
where each vertex has a color among a set of size $q$,
each edge has a weight at most $r = \max_{i,j} R_{i,j}$
and each connected component a weight $\sigma$,
\begin{equation} \label{eq:major}
  g^{(k)}_{R,\sigma}(n,m) \leq q^n r^m \max(\sigma, 1)^n g^{(k)}_{(1),1}(n,m).
\end{equation}
Combining Equations~\eqref{eq:minor} and~\eqref{eq:major},
we conclude that there exist two constants $C_2$ and $C_3$,
independent of $n$ and $k$, such that for $n$ large enough,
\[
  \frac{g^{(k)}_{\cm,\sigma}(n,m)}{g^{(0)}_{\cm,\sigma}(n,m)}
  \leq
  C_2 (C_3 r^{m/n})^n 
  \frac{g^{(k)}_{(1),1}(n,m)}{g^{(0)}_{(1),1}(n,m)}
\]
where $r$ is the maximum of the coefficients of $\cm$.

Since Theorem~\ref{th:complexpart1} is equivalent with Equation~\eqref{eq:complexpart1sum}
and holds for classic multigraphs, the previous inequality proves that
\[
  \lim_n \sum_{k \geq 0} \frac{g^{(k)}_{\cm,\sigma}(n,m)}{g^{(0)}_{\cm,\sigma}(n,m)} = 0
\]
as soon as $C_3 r^{m/n}$ is smaller than~$1$, 
i.e. for matrices~$\cm$ with small enough coefficients.
But Theorem~\ref{th:complexpart1} is independent of the size of the coefficients of~$\cm$,
because this matrix can be replaced by $\alpha \cm$ for any positive $\alpha$ without
changing the structure of the graphs.
\end{proof}

\begin{theorem} \label{th:complexpart}
With the notations of Theorem~\ref{th:main_vt},
when~$m = \frac{n}{2} ( 1 + \mu n^{-1/3} )$
and~$|\mu| \leq n^{1/12}$,
the asymptotic number of $(\cm,\sigma)$-multigraphs
with complex part of excess~$k$ is
\[
  g^{(k)}_{\cm,\sigma}(n,m)
  \sim
  \frac{n^{2m}}{2^m m!}
  \frac{(\sigma q)^{n-m}}{\chi(1)^{\sigma/2}}
  n^{(\sigma-1)/6} 
  \sigma^k e^{(\sigma)}_k
  \sqrt{2 \pi} A\left( 3k + \frac{\sigma}{2}, \mu \right).
\]
\end{theorem}

\begin{proof}
In Equation~\eqref{eq:subcritical}, there are two saddle-points
that are distinct when $m/n < 1/2$, but coalesce at this critical value.
In this context, the large powers scheme ceases to apply,
so we replace it with~\cite[Lemma $3$]{JKLP93} to obtain Equation~\eqref{eq:gk}
(see also~\cite[Theorem~11]{BFSS01} 
and~\cite[Theorem~IX$.16$]{FS09} for links 
with the \emph{stable laws} of probability theory).
This lemma computes asymptotics of the shape
\[
  [z^n] \frac{U(z)^{n-m}}{(1-T(z))^y}
\]
where~$y$ is a real constant.
In particular, it proves that for any real fixed real values $y_1$ and $y_2$,
\[
  [z^n] \frac{U(z)^{n-m}}{(1-T(z))^{y_1}}  
  \sim n^{(y_1 - y_2)/3} [z^n] \frac{U(z)^{n-m}}{(1-T(z))^{y_2}}.
\]
Therefore, in Equation~\eqref{eq:gk},
the only term of~$K_{k,\sigma}(z)$ that influence the asymptotic
is given by~\eqref{eq:Kcubic}. We then apply Lemma~3 of~\cite{JKLP93} to
\[
  g^{(k)}_{\cm,\sigma}(n,m)
  \sim
  \frac{(n-m)!}{(n-m+k)!}
  \frac{(\sigma q)^{n-m}}{\chi(1)^{\sigma/2}}
  \frac{\sigma^k |\cubics_{k,\sigma}|}{2^k (2k)!}
  \frac{n!}{(n-m)!}
  [z^n]
  \frac{\left( T(z) - \frac{T(z)^2}{2} \right)^{n-m}}{(1-T(z))^{3k+\sigma/2}}.
\]
\end{proof}
Theorem~\ref{th:main_vt} is then established by summation of the~$g^{(k)}_{\cm,\sigma}(n,m)$.
The result is multiplied by~$\delta^m$,
$\sigma$ is replaced by~$c \sigma$, $q$ by~$q/c$ and $\chi(X)$ is adjusted.
More information on the analytic behavior of~$A(y,\mu)$
can be found in~\cite[Lemma~3]{JKLP93}.

    \section{Conclusion}

We have presented a model of random multigraphs 
with colored vertices and weighted edges,
similar to the inhomogeneous graph model~\cite{S02}.
Using tools developed in~\cite{JKLP93}, we derived 
a complete picture of the finite size scaling 
and the critical exponents associated 
to the birth of complex components.
Applications to bipartite graphs and to the satisfiability of quantified~$2$-Xor-formulas 
raised new proof of known results~\cite{PY10} and new results.

In this paper, the emphasis is on the link between
the birth of complex components in~$(\cm,\sigma)$-multigraphs
and the phase transition of tractable satisfiability problems.
This justifies the restriction to vertex-transitive matrices~$\cm$,
often encountered in applications, and the addition
of the factor~$\sigma$ to the original inhomogeneous graph model
in order to enrich the expressiveness.

The present results can be extended to simple $\cms$-graphs.
Indeed, almost surely, all loops and multiple edges
of the random $\cms$-multigraphs considered belong to unicyclic components.
So the only adjustment needed is to replace the generating function~$V(z)$
with~$V(z) - \frac{1}{2} \sum_i \cm_{i,i} T_i(z) - \frac{1}{4} \sum_{i,j} \cm_{i,j}^2 T_i(z) T_j(z)$.
Due to the lack of space, the proof of this result is not included.

We now plan to extend our result to non-vertex-transitive matrices~$\cm$,
and to enumerate $\cm$-multigraphs with a larger density of edges.

\bibliographystyle{amsplain}
\bibliography{/home/elie/research/articles/bibliography/biblio}
\end{document}